\documentclass[12pt]{amsart} 
\usepackage{url} 
\usepackage{verbatim} 
\usepackage{amssymb} 
\usepackage{upref}
\usepackage[all]{xy} 

\allowdisplaybreaks

\newtheorem{thm}{Theorem}[section]
\newtheorem{prop}[thm]{Proposition}

\theoremstyle{definition} 

\newtheorem{rem}[thm]{Remark}

\newcommand{\CC}{\mathcal C} 
\newcommand{\MM}{\mathcal M}
\newcommand{\NN}{\mathcal N} 

\DeclareMathOperator{\obj}{Obj} 

\newcommand{\dn}{\downarrow}

\begin{document}
\title{On reflective-coreflective equivalence and associated pairs}
\author{Erik B\'edos}
\address{Institute of Mathematics, University of Oslo, P.B. 1053 
Blindern, 0316 Oslo, Norway}
\email{bedos@math.uio.no}
\author{S. Kaliszewski}
\address{School of Mathematical and Statistical Sciences, Arizona 
State University, Tempe, AZ 85287}
\email{kaliszewski@asu.edu}
\author{John Quigg}
\address{School of Mathematical and Statistical Sciences, Arizona 
State University, Tempe, AZ 85287}
\email{quigg@asu.edu}
\date{December 20, 2011} 



\hyphenation{sub-categories}

\begin{abstract} We show that a reflective/coreflective pair of full
subcategories satisfies a ``maximal-normal''-type equivalence if and
only if it is an associated pair in the sense of Kelly and Lawvere.
\end{abstract}

\maketitle

\section{Introduction}

In a recent paper \cite{BKQ} we explored a special type of category
equivalence between reflective/coreflective pairs of subcategories that
we first encountered in the context of crossed-product duality for
$C^*$-algebras.  Because our main example of this phenomenon involved
categories of maximal and normal $C^*$-coactions of locally compact
groups, we called it a ``maximal-normal''-type equivalence.

Since then, F.~W.~Lawvere has drawn our attention to~\cite{KL}, where
G.~M.~Kelly and he introduced the concept of \emph{associated pairs} of
subcategories. The purpose of this short note is to show that these two
notions of equivalence are the same: a reflective/coreflective pair of
full subcategories satisfies the ``maximal-normal''-type equivalence
considered in \cite{BKQ}  if and only if it is an associated pair in the
sense of~\cite{KL}.

As operator algebraists, we had hoped with~\cite{BKQ} to initiate a
cross-fertilization between operator algebras and category theory, and
we are grateful to Ross Street for the role he has played in helping
this happen.  Our understanding of the operator-algebraic examples has
certainly been deepened by this connection; ideally, the techniques and
examples  of ``maximal-normal''-type equivalence will in turn provide a
way of looking at associated pairs that will also be useful to category
theorists.

\section{Maximal-normal equivalences and associated pairs}

Our conventions regarding category theory follow \cite{maclane}; see
also  \cite{BKQ}. Throughout this note, we let $\MM$ and $\NN$  denote
full subcategories of a category $\CC$, with $\NN$ reflective and $\MM$
coreflective. The inclusion functors  $I : \MM \to \CC$ and
$J:\NN\to\CC$ are then both full and faithful. 
We also use the
following notation:

\begin{itemize} 
\item $N\colon\CC\to\NN$ is a reflector and
$\theta\colon1_\CC\to J   N$ denotes the unit of the adjunction $N\dashv
J$; 

\item $M\colon\CC\to\MM$ is a coreflector and   $\psi\colon I  M\to
1_\CC$ denotes the counit of the adjunction $I \dashv M$. 
\end{itemize} In
\cite[Corollary 4.4]{BKQ} we showed that the adjunction $N  I \dashv M 
J $ is an adjoint equivalence between $\MM$ and $\NN$ if and only if
\begin{enumerate} 
\item[(I)] 
for each $y\in\obj\NN$, $(y,\psi_y)$ is an
{initial} object in the comma category $My\dn\NN$; and 
\item[(F)] 
for
each $x\in\obj\MM$, $(x,\theta_x)$ is a {final} object in the comma
category $\MM\dn Nx$. 
\end{enumerate}

In all our examples in \cite{BKQ}, the adjoint equivalence $N  I \dashv
M  J $ between $\MM$ and $\NN$ was what we called the ``maximal-normal"
type (recall that this terminology was motivated by the particular example of maximal and normal coactions on $C^*$-algebras; see \cite[Corollary~6.16]{BKQ}):  in addition to (I) and (F), such an adjunction satisfies
\begin{enumerate} 
\item[(A)] 
for each $z\in\obj\CC$, $(Nz, \theta_z
\circ \psi_z)$ is an initial object in  $Mz\dn \NN$. \end{enumerate}
Equivalently, by \cite[Theorem 3.4]{BKQ}, (I) and~(F) hold, and
\begin{enumerate} 
\item[(B)] 
for each $z\in\obj\CC$, $(Mz,
\theta_z\circ\psi_z)$ is a final object in $\MM\dn Nz$. 
\end{enumerate}
In fact, conditions~(A) and (B) alone suffice:

\begin{prop}\label{max-nor} 
The adjunction $N  I \dashv M  J $ between
$\MM$ and $\NN$ is a ``maximal-normal" adjoint equivalence if and only
if \textup(A\textup) and \textup(B\textup) hold. 
\end{prop}
\begin{proof} 
By \cite[Theorem 4.3]{BKQ}, (I) is equivalent to
\begin{enumerate} 
\item[(I$'$)] 
for each $y \in \obj \NN,\,  N\psi_y:NMy
\to Ny$  is an isomorphism, 
\end{enumerate} 
while (F) is equivalent to
\begin{enumerate} 
\item[(F$'$)] 
for each $x \in \obj \MM,\, 
M\theta_x:Mx \to MNx$  is an isomorphism. 
\end{enumerate} 
On the other
hand, by \cite[Theorem 3.4]{BKQ}, (A) is equivalent to 
\begin{enumerate}
\item[(A$'$)] 
for  each $z \in \obj \CC$,  $N\psi_z$ is an isomorphism,
\end{enumerate} 
while (B) is equivalent to 
\begin{enumerate}
\item[(B$'$)] 
for  each $z \in \obj \CC$,  $M\theta_z$ is an
isomorphism. 
\end{enumerate} 
Now clearly, (A$'$) implies (I$'$) and
(B$'$) implies (F$'$), so (A) implies~(I) and~(B) implies~(F).
\end{proof}

We now recall from \cite{CHK, KL} that a morphism~$f$ in~$\CC(x,y)$ and
an object~$z$ of~$\CC$ are said to be \emph{orthogonal} when the map
$\Phi_{f,z}$ from~$\CC(y, z)$ into $\CC(x,z)$ given by $\Phi_{f,z}(g) =
g\circ f$ is a bijection. The collection of  all morphisms in $\CC$
that are orthogonal to every object of $\NN$ is denoted by $\NN^\perp$.

As shown in \cite[Proposition 2.1]{KL}, a morphism $f: x\to y$ in $\CC$
belongs to $\NN^\perp$ if and only if $f$ is inverted by $N$, that is,
$Nf$ is an isomorphism. (The standing assumption in \cite{KL} that $\NN$
is replete is not necessary for this fact to be true. To see this, note
that $Nf$ is an isomorphism if and only if the map $\Psi_{f,z}$  from
$\NN(Ny, z)$ into $\NN(Nx, z)$ given by  $\Psi_{f,z}(h) = h \circ Nf$ is
a  bijection for each object $z$ of $\NN$.  For each such $z$, the
universal properties of $\theta$ imply that the map $\tau_{w,z}$  from
$\NN(Nw, z)$ into $\CC(w,z)$ given by $\tau_{w,z}(g) = g \circ
\theta_{w}$ is a bijection for each object $w$ of $\CC$. Now, as $
\theta_y \circ f =   Nf  \circ \theta_x$, the diagram 
\[ 
\xymatrix{
\NN(Ny, z) \ar[r]^-{\Psi_{f,z}} \ar[d]_{\tau_{y,z}} 
&\NN(Nx, z) \ar[d]^{\tau_{x,z}} 
\\ 
\CC(y,z) \ar[r]_-{\Phi_{f,z}} 
&\CC(x,z) } 
\] 
is
readily seen to commute. It follows that $\Psi_{f,z}$ is a bijection if
and only if $\Phi_{f,z}$ is a bijection. This shows that $Nf$ is an
isomorphism if and only if $f$ is orthogonal to $z$ for each object $z$
of $\NN$, \emph{i.e.}, if and only if $f$ belongs to $\NN^\perp$.)

Similarly, a morphism~$f$ in $\CC(x,y)$  and an object~$z$ in~$\CC$ are
\emph{co-orthogonal}  when the map $g \to f \circ g$ from $\CC(z, x)$
into $\CC(z,y)$ is a bijection. The collection of  all morphisms in
$\CC$ that are co-orthogonal to every object in $\MM$ is denoted by
$\MM^\top $. Equivalently,  a morphism $f: x\to y$ in $\CC$ belongs to
$\MM^\top $ if and only if $f$ is inverted by $M$, that is, if and only
if $Mf$ is an isomorphism.

The pair $(\NN, \MM)$ is called an \emph{associated pair} if $\NN^\perp
= \MM^\top $; equivalently, if for every morphism $f$ in $\CC$, $N$
inverts $f$ if and only if $M$ does.  We refer to \cite[Section 2]{KL}
for more information concerning this concept (in the case where both
$\MM$ and $\NN$ are also assumed to be replete).

\begin{thm}\label{assmaxnor} 
The adjunction $N  I \dashv M  J $ is a
``maximal-normal" adjoint equivalence if and only if $(\NN,\MM)$ is an
associated pair. 
\end{thm}

\begin{proof} 
First assume that $(\NN,\MM)$ is an associated pair, and
let $x$ be an object in $\CC$.  As pointed out above, the map
$\tau_{x,z}$ is a bijection from $\NN(Nx,z)$ into $\CC(x,z)$ for each
object $z$ of $\NN$. But $\Phi_{\theta_x,z} = \tau_{x,z}$, so this means
that $\theta_x$ lies in $\NN^\perp$, and therefore in $\MM^\top $. As
$\MM^\top $ consists of the morphisms in $\CC$ that are inverted by $M$,
we deduce that $M\theta_x$ is an isomorphism.  This shows that (B$'$)
holds, and therefore that~(B) holds.  The argument that (A) holds is
similar, so $N  I \dashv M  J $ is a ``maximal-normal" adjoint
equivalence by Proposition \ref{max-nor}.

Now assume that the adjunction $N  I \dashv M  J $ is a
``maximal-normal" adjoint equivalence.  Then $N\cong NIM$ by
\cite[Proposition~5.3]{BKQ}, and $NI$ is an equivalence.  So for any
morphism~$f$ of~$\CC$, we have 
\begin{align*} 
\text{$Nf$ is an
isomorphism} 
&\Leftrightarrow\text{$NIMf$ is an isomorphism}\\
&\Leftrightarrow\text{$Mf$ is an isomorphism.} 
\end{align*} 
Thus
$(\NN,\MM)$ is an associated pair. 
\end{proof}

\begin{rem} 
In the examples presented in \cite[Section 6]{BKQ}, the
adjunctions $N  I \dashv M  J $ are ``maximal-normal" adjoint
equivalences, so all the pairs $(\NN, \MM)$ there are associated pairs.
Moreover, all these pairs consist of subcategories that are easily seen
to be replete. It follows from  \cite[Theorem 2.4]{KL} that $\MM$ and
$\NN$ are uniquely determined 
as subcategories
by each other, a fact that is not \emph{a
priori} obvious in any of the examples. 
\end{rem}

\end{document}